\documentclass[twoside]{article}
\usepackage[a4paper, left=3.5cm, right=3.5cm, top=3cm, bottom=3cm]{geometry}
\usepackage{graphicx} 
\usepackage[final]{showlabels}
\usepackage[backend=biber,style=numeric]{biblatex}
\usepackage{amsmath,amssymb,amsthm,yhmath,amsfonts,mathrsfs,enumitem,mathtools,dsfont}

\usepackage{tikz}
\usetikzlibrary{automata,shapes,arrows,positioning}

\newcommand{\pr}{\mathbb{P}}
\newcommand{\E}{\mathbb{E}}

\newcommand{\dd}{\mathop{}\!\mathrm{d}}

\newcommand{\1}[1]{{\mathds{1}}_{\left\{#1\right\}}}

\newtheorem{lemma}{Lemma}
\newtheorem{remark}{Remark}
\newtheorem{theorem}{Theorem}
\newtheorem{definition}{Definition}
\newtheorem{corollary}{Corollary}

\newcommand{\keywords}[1]{\noindent\textbf{Keywords:} #1}
\newcommand{\subjclass}[1]{\noindent\textbf{AMS2020 subject classification:} #1}

\usepackage{fancyhdr}
\usepackage[colorlinks=true, linkcolor=blue, citecolor=red, urlcolor=magenta]{hyperref}
\usepackage{orcidlink}

\title{The effects of initial conditions on the accuracy of mean-field approximations of Markov processes on large random graphs}

\usepackage{authblk}

\author[1]{Pierfrancesco Dionigi\,\orcidlink{0000-0003-2180-8669}\thanks{Pierfrancesco Dionigi is supported by the ERC Synergy under Grant No. 810115.\\ \texttt{email: dionigi@renyi.hu}}}
\author[1,2]{Dániel Keliger\,\orcidlink{0000-0003-4981-098X}\thanks{Dániel Keliger is partially supported by the ERC Synergy under Grant No. 810115 - DYNASNET and NKFI-FK-142124. \\ \texttt{email: zunerd@renyi.hu} } }
\affil[1]{HUN-REN Alfréd Rényi Institute of Mathematics, Budapest, Hungary}
\affil[2]{Department of Stochastics, Budapest University of Technology and Economics, Hungary}

\date{\today}

\begin{document}
\maketitle

\begin{abstract}
We study the evolution of a general class of stochastic processes (containing, e.g. SIS and SIR models) on large random networks, focusing on a particular general class of random graph models. To approximate the expected dynamics of these processes, we employ a mean-field technique known as the \emph{N-Intertwined Mean-Field Approximation} (NIMFA) \cite{vanmieghem2009virus}. Our primary goal is to quantify the impact of the randomness in the network topology on the accuracy of such approximations. While classical work by Kurtz \cite{kurtz1978strong} established strong approximation results for density-dependent Markov chains using  mean-field on the complete-graph of order $N$, yielding error bounds of order \( \frac{1}{\sqrt{N}} \), we demonstrate that in the context of NIMFA on random graphs, the error admits a refined characterization depending on initial conditions. Specifically, for generic initial conditions, the worst case error is of order \( \frac{1}{\sqrt{d}} \),  where \( d \) is the expected average degree of the graph. This result highlights how network sparsity contributes an additional source of variability not captured by classical mean-field approaches. Furthermore, when the initial conditions are taken to be fairly homogeneous, we show that the error term is of order \( \frac{1}{\sqrt{N}}+\frac{1}{d}\), where \( N \) is the number of nodes, underscoring the intricate role of the choice of initial conditions in the error bounds. Our analysis connects with recent interest in quantifying the role of graph heterogeneity in epidemic thresholds and dynamics \cite{chakrabarti2008epidemic, pastor2015epidemic, demirel2017disease}.
\end{abstract}
\\
\section{Introduction}

The study of stochastic processes on networks is central to understand the spread of epidemics, information, and behaviors in complex systems. This in turn prompted in the recent years a prolific research aimed to understand the underlying geometrical structure governing these complex systems and how it influences dynamics. This was one of the main drives that triggered the boom of network science during the last 30 years: networks indeed represent a pivotal tool in the study of systems constructed by complex interactions. The variety of areas where networks have been adopted is huge; their application in biology, economics, sociology, industrial applications and many more different areas required the birth of a very diversified range of \emph{networks models} (or \emph{ensembles} in physics literature) that are able to capture specific features of the real world system they try to represent \cite{newman2003structure,barabasi2016network,chuang2010decade,schweitzer2009economic}. Together with the classic Erd\H{o}s-Rényi random graph is worth to mention, Chung-Lu random graphs, Inhomogeneous random graphs, Stochastic Block model, Preferential Attachment models and Maximal Entropy random graphs \cite{bollobasPhaseTransitionInhomogeneous2007,ChungLu, BAnetwork, random_geometric_network,squartiniMaximumEntropyNetworksPattern2017}. A lot of research focused as well on how the main properties of these graph models might affect the dynamical behavior, such as spectral properties, degree fluctuations, connectivity properties and many more \cite{MCKEE2024, graphon_spin, random_walk_evolving_set,dionigiSpectralSignatureBreaking2021}. 
On the other hand, there is a complete opposite driving force in the study of stochastic processes: for most of them the exact solution of the related master equation is unfeasible. Furthermore, most of the times our ability to solve (or at least give an attempt) these problems heavily depends on the set of possible initial conditions, leading to leading to three layers of randomness. Therefore, over the years, different approaches relying either on assumptions that lead to simplifications or on approximations were attempted. Most of the times these approaches rely on averaging the noise provided by the background randomness of the graph, or try to plug in independence conditions in order for some equations to decouple, provided that the initial conditions respect certain conditions; collectively these approaches are called \emph{mean-field} models. These approximations often reduce the complexity of the system by replacing stochastic interactions with deterministic equations that describe average behavior. Therefore, each of these approaches, no matter how sophisticated, introduce approximations, and, as with every approximation, we pay an an error in the accuracy of our model. 
Given the richness in the choice of the underlying graph model, in the possible initial conditions, and in the ways in which the processes can be approximated, is of central interest understanding the error we pay for different choices and different approximations. This translates into trying to figure out general upper and lower bounds that assure how much this error can be and how it can depend from the possible choices of graph ensembles, initial conditions, and process' dynamics. This question relates closely to the classical work by Kurtz \cite{kurtz1978strong}, who developed strong approximation theorems for density-dependent Markov chains. In his framework, the network structure is implicitly averaged out, effectively modeling interactions as occurring on a weighted complete graph. Under this assumption, he derived an error bound of order \( \frac{1}{\sqrt{N}} \), where \( N \) is the size of the system. Despite penalizing the network geometry, the seminal idea of Kurtz has been for long time the main result in this area. 
Attempts have been made to indeed plug in a more refined account for the network geometry. For example, in \cite{guerraAnnealedMeanfieldFormulations2010}, the authors show how to implement an Heterogeneous Mean-Field Approximation(HMFA) is effective to take into account the annealed version of the underlying random graph model, i.e. the average adjacency matrix.

In this paper, we investigate the performance of NIMFA, the \emph{N-Intertwined Mean-Field Approximation}, when applied to stochastic processes on a fairly general class of random graphs that are able to model real world communities. This technique has gained attention due to its tractability and improved accuracy over naive homogeneous and heterogeneous mean-field models, particularly when applied to heterogeneous networks, which are not necessarily dense, but have large degrees \cite{vanmieghem2009virus,10.1145/3508033, Sridhar2023}. Our focus is on understanding how the randomness of the underlying network model influences the accuracy of this approximation, particularly as the graph size and average degree vary. In contrast with the work of Kurtz, the NIMFA approach retains the individual network structure and thus introduces an additional source of randomness due to the graph itself, i.e. the graphs can be considered \emph{quenced} without any necessary averaging procedure (\emph{annealing}). Our main contribution is to rigorously quantify this additional randomness source and its influence on the error term. Our work relies on the results of \cite{keligerConcentrationMeanField2024a}. There the authors show a tight bound of order \( \frac{1}{\sqrt{N}} + \frac{1}{d} \) where $d$ is the average degree of the quenched graph. In our work (Theorems \ref{t:general_homogeneous_upper_bound} and \ref{t:homogeneous_lower_bound}), in the special case when the initial conditions of the process are \emph{fairly} homogenous, we show that the approximation error for the \emph{annealed NIMFA}, i.e. the HRF obtained by annealing the graph model after applying the NIMFA, is of the same order, \( \frac{1}{\sqrt{N}} + \frac{1}{d} \), where \( d \) is the expected average degree of the network. This result demonstrates that in sparse networks, where \( d \ll \sqrt{N} \), the randomness of the graph has a non-negligible effects on the mean-field accuracy, and must be accounted for in modeling and inference tasks. Furthermore for general initial conditions, surprisingly, the worst case error is of order \(\frac{1}{\sqrt{d}}\). This error is much larger than the one obtained for homogenous initial conditions when \( d \ll N \). 

The rest of the paper is organized as follows. In Section \ref{sec:setup}, we introduce the setup necessary to state our main results. In Section \ref{sec:results} we state our main results and in Section \ref{sec:proofs} we present the proofs or our theorems. Finally in Section \ref{sec:conclusion} we discuss implications, limitations, and directions for future work.

\section{Setup}\label{sec:setup}
In this section we will introduce the main mathematical objects we will work with and will be needed for stating our results in section \ref{sec:results}. In \ref{sec:graph_model} we introduce the class of graph models we will work with. In \ref{sec:stoch} we enter in the details of stochastic processes and the mean-field technique we will use, NIMFA, together with our annealed approximation; we will introduce the quantity we would like to approximate and we finally provide a couple of examples of familiar models to give a better understanding of our techniques.
\subsection{The graph model}\label{sec:graph_model}

We use the stochastic block model to generate inhomogeneous graphs. The number of vertices is denoted by $N$ and the adjacency matrix is $A=\left(a_{ij} \right)_{i,j=1}^N$. The vertices are partitioned into
\begin{align}\label{eq:partition}
    \bigsqcup_{k=1}^{K} I_k=\{1,\dots,N \}.
\end{align}
The sizes are assumed to be "macroscopic", that is, there is a constant $0<c_{block}<1$ such that
\begin{align*}
    N_k:=\left |I_k \right| \geq c_{block}N.
\end{align*}
The fraction of vertices belonging to the $k$th partition is
\begin{align*}
\pi_{k}:= \frac{N_k}{N} \geq c_{block}.
\end{align*}
Each edge $ij$ is drawn independently with probability 
\begin{align*}
   p_{ij}:= \pr \left(a_{ij}^N=1 \right)=\rho w_{kl} \ \left( i \in I_k, \ j \in I_l, \ i \neq j \right), 
\end{align*}
  where $\rho$ is a density parameter and the weights are assumed to be bounded: $w_{kl} \leq M.$

Without the loss of generality we may assume that there is a constant $m$ such that
\begin{align}
\label{density_cond}
 m \leq \sum_{k,l=1}^K w_{kl} \pi_k \pi_l \leq M   
\end{align}
as we are not interested in the case when all the weights are zero and we could modify $\rho$ otherwise.

The realized and expected degrees are
\begin{align*}
\delta_i:=&\sum_{j=1}^N a_{ij} \\
d_i:=&\sum_{j=1}^N p_{ij}.
\end{align*}
The expected average degree is 
\begin{align}
\label{eq:d}
d:= \frac{1}{N}\sum_{i=1}^{N}d_i=N \rho \sum_{k,l=1}^K w_{kl}\pi_k \pi_l+O(1)=\Theta (N \rho),
\end{align}
where $O(1)$ controls the diagonal terms $i=j$.

We also define $D:=\max_i{d_i}$ the maximal expected degree. This will be realized within one or multiple blocks.

To simplify the notation we sometimes work with the normalized adjacency matrix
\begin{align*}
    \frac{1}{N \rho}A=:B=(b_{ij})_{i,j=1}^N.
\end{align*}
Furthermore, we introduce
$\hat{B}:=\E(B)+\frac{1}{N \rho} \mathcal{D}$ where $\mathcal{D}$ is a diagonal matrix with entries
\begin{align*}
\mathcal{D}_{ii}=\rho w_{kk} \ (i \in I_k).
\end{align*}

We are mostly interested in the densities where $d$ is at least a positive power of $N$, that is, we assume there is an $0<\alpha \leq 1$ such that
\begin{align*}
d \geq N^{\alpha}.    
\end{align*}
For a random variable $X$ we denote its second moment by
\begin{align*}
\|X \|^2:= \E(X^2).
\end{align*}
The indexless norm notation is only used for scalar valued random variables. For $N$ dimensional vectors the $L^p$ norms - and their induced operator norms - are denoted by $\| \cdot \|_p$ and they do not include expectations.

\subsubsection{Spectral properties}
Following \cite[Theorem 3.2]{benaych-georgesSpectralRadiiSparse2020} we have that, given $H=D^{-1/2}(A-\E A)$ and given the parameters 
\[
\kappa= \rho \frac{N\max w_{kl}}{D}=\Theta(1)\qquad q=\sqrt{D}\wedge N^{1/10} \kappa^{-1/9} 
\]
(i.e, if $D\ll N^{1/5}$, $q=\sqrt{D}$, while $q=N^{1/10}\kappa^{-1/9}$ if $D\gg N^{1/5}$), we have 
\[
\frac{\E \|A-\E A\|_2}{\sqrt{D}}\leq 2 + O\left(\frac{\log N}{q} \right).
\]

Note that
\begin{align*}
 q=&\Theta \left(N^{\beta} \right), \ \textit{where} \ \beta:= \min \left\{\frac{\alpha}{2} , \frac{1}{10} \right\}.
\end{align*}

We see that the normalization in the above formula is different from the one adopted in the definition of $B$. But given condition \ref{density_cond} we have that $D=\Theta(N\rho)$, therefore $\E \left(\|B-\E B\|_2 \right)=O \left(\frac{1}{\sqrt{d}} \right)$.

Using the result in \cite[Examples 3.14 and 6.8]{boucheronConcentrationInequalitiesNonasymptotic2013b} we have that 
\begin{equation}
\pr\left( \left | \|H\|_2-\E \|H\|_2\right|>t \right)=\leq 2e^{-cq^2t^2}
\end{equation}
where c is a constant.

Clearly,
\begin{align*}
\|B-\hat{B} \|_2 \leq \|B-\E(B) \|_2+\frac{1}{N \rho}\|\mathcal{D} \|_2  \leq \|B-\E(B) \|_2+\frac{M}{N}. 
\end{align*}
Therefore we have the following: 
\begin{corollary}
\label{c:spectral}
There are constants $c_1,c_2>0$ such that
\begin{align}
\label{eq:spectral_concentration}
\pr \left( \|B-\hat{B} \|_2 > \frac{c_1}{\sqrt{d}} \right) \leq 2 e^{-c_2 N^{2\beta}}.
\end{align}
\end{corollary}
\subsection{Stochastic dynamics on graphs}\label{sec:stoch}

A vertex can take one state from the finite state space $\mathcal{S}.$ The indicator that vertex $i$ is in state $s$ at time $t$ is denoted by $\xi_{i,s}(t).$ We will assume that the initial conditions $\xi_{i,s}(0)$ are \emph{conditionally }independent for different $i$'s with respect to $G$. That is, we allow the initial conditions to depend on the concrete realization of the random graph $G$, but all other source of randomness for the nodes must be independent.

In order to separate the randomness arising from the dynamics and the one coming from the graph model we will use the conditional probabilities, expectations and variances:
$\pr_{G}(\cdot):=\pr \left( \cdot \left | G \right. \right), \ \E_{G}(\cdot):=\E \left( \cdot \left | G \right. \right)$ and $\operatorname{Var}_G(\cdot):= \operatorname{Var}(\cdot|G)$ respectively.

If vertex $i$ is in state $s$ then via a self interaction it can transition to state $s'$ at rate 
$$q_{s \to s'}.$$
Furthermore, if vertex $j$ is in state $\tilde{s}$ and vertex $i$ is in state $s$ than vertex $j$ makes vertex $i$ transition to state $s'$ at rate
$$q_{\tilde{s};s  \to s'}\frac{a_{ij}}{N \rho}.$$ Note that
\begin{align*}
\frac{1}{N \rho}\sum_{j=1}^{N}a_{ij}=\frac{\delta_i}{N \rho}\overset{\eqref{eq:d}}{=} \Theta \left( \frac{\delta_i}{d} \right)
\end{align*}
which, in our framework, remains $O(1)$ (see Lemma \ref{l:degree_bound}). Therefore, with this normalization, we ensure that in this time scale there are at most a constant amount of transitions per unit time.

We make use of the convention
\begin{align*}
q_{ s \to s}:=&-\sum_{\substack{s' \in \mathcal{S} \\ s' \neq s}}q_{s \to s'} \\
q_{\tilde{s}; s \to s}:=&-\sum_{\substack{s' \in \mathcal{S} \\ s' \neq s}}q_{\tilde{s}; s \to s'}    
\end{align*}
to denote the total (negative) rate of vertex $i$ leaving state $s$.

By using the infinitesimal generator of the stochastic process we get the following identity
\begin{align}\label{eq:ODE_non_closed}
\frac{\dd}{\dd t} \E_G \left(\xi_{i,s}(t) \right) =&\sum_{s' \in \mathcal{S}}q_{s'\to s}\E_{G} \left(\xi_{i,s'}(t) \right)\\
&+\sum_{s', \tilde{s} \in \mathcal{S}}q_{\tilde{s}, s' \to s}  \frac{1}{N \rho}\sum_{j} a_{ij} \E_G \left( \xi_{i,s'}(t) \xi_{j,\tilde{s}}(t) \right).
\end{align}

One can close \eqref{eq:ODE_non_closed} by artificially imposing independence between the variables which results in
\begin{align}
\label{eq:NIMFA}
\begin{split}
   \frac{\dd}{\dd t} z_{i,s}(t) =& \sum_{s' \in \mathcal{S}}q_{s'\to s} z_{i,s'}(t)\\
   &+\sum_{s', \tilde{s} \in \mathcal{S}}q_{\tilde{s}, s' \to s}  \frac{1}{N \rho}\sum_{j} a_{ij}  z_{i,s'}(t) z_{j,\tilde{s}}(t), \\
   z_{i,s}(0)=&\pr_G(\xi_{i,s}(0)=1)
\end{split}
\end{align}
If the correlation between indicators are negligible than we should have $\pr_G\left(\xi_{i,s}(t)=1 \right) \approx z_{i,s}(t).$ Such method is called quenched mean-field approximation 
or N-Interwined Mean-Field Approximaton (NIMFA) \cite{vanmieghem2009virus}. It is more convenient to write \eqref{eq:NIMFA} in the form:
\begin{align}
\label{eq:NIMFA2}
\begin{split}
   \frac{\dd}{\dd t} z_{i,s}(t) =& \sum_{s' \in \mathcal{S}}q_{s'\to s} z_{i,s'}(t)\\
   &+\sum_{s', \tilde{s} \in \mathcal{S}}q_{\tilde{s}, s' \to s}    z_{i,s'}(t) \left(Bz_{\cdot,\tilde{s}}(t) \right)_{i}, \\
   z_{i,s}(0)=&\pr_G(\xi_{i,s}(0)=1)
\end{split}
\end{align}
where $z_{\cdot,s}(t):=\left(z_{i,s}(t) \right)_{i=1}^{N}.$

In real world settings, neither the underlying graph nor the states of individual vertices are known; however, we may observe some mesoscopic quantities on the partition of the graph that might correspond to regions or age groups among other things, which justifies the class of graph models we chose.

The expected number of edges between a vertex $i \in I_k$ and the partition $I_l$ is
\begin{align*}
\sum_{j \in I_l}p_{ij}=N \rho w_{kl} \pi_l.
\end{align*}
Replace $a_{ij}$ with the expectation $p_{ij}$ ( with $p_{ii}=\rho w_{kk}$ for $i \in I_k$ in this particular case) in \eqref{eq:NIMFA} and assume $z_{i,s}(t)=x_{k,s}(t)$ for all $i \in I_k.$ Then the ODE system for $x_{k,s}(t)$ becomes
\begin{align}
\label{eq:x}
\begin{split}
   \frac{\dd}{\dd t} x_{k,s}(t) =& \sum_{s' \in \mathcal{S}}q_{s'\to s} x_{k,s'}(t)\\
   &+\sum_{s', \tilde{s} \in \mathcal{S}}q_{\tilde{s}, s' \to s}  \sum_{j} w_{kl}\pi_l  x_{k,s'}(t) x_{l,\tilde{s}}(t), \\
   x_{k,s}(0)=& \frac{1}{N_k}\sum_{i \in I_k} \pr_G \left(\xi_{i,s}(0)=1 \right).
\end{split}
\end{align}
We will call $x_{k,s}(t)$ the Block Homogenous Mean-Field Approximation (BHMFA).

Our main goal is to give a quantitative error bound on the error between 
\begin{align}
\label{eq:xi_bar}
\bar{\xi}_{k,s}(t):=\frac{1}{N_k}\sum_{i \in I_k} \xi_{i,s}(t)
\end{align}
and $x_{k,s}(t)$.

\subsubsection{Examples}\label{subsec:examples}

\subsubsection*{The SIR process}
One of the main motivation for studying Markov processes on large graphs is to understand the spread of diseases on social networks. 

A frequently used model is the so-called SIR process. The state space is $\mathcal{S}=\{S,I,R\}$ where the initials stand for susceptible, infected and recovered.

There are two kinds of transitions in the system. An infected vertex can infect a susceptible neighbor at rate $\frac{\beta}{N \rho}$ (unrelated to the $\beta$ in Corollary \ref{c:spectral}), while said infected node recovers at rate $\gamma$. These translate to $q_{I, S \to I}=\beta$ and $q_{I \to R}=\gamma.$

For this process NIMFA takes the form of
\begin{align*}
    \frac{\dd}{\dd t}z_{i,S}(t)=&-\frac{\beta}{N \rho}z_{i,S}(t) \sum_{j}a_{ij}z_{j,I}(t) \\
    \frac{\dd}{\dd t}z_{i,I}(t)=&\frac{\beta}{N \rho}z_{i,S}(t) \sum_{j}a_{ij}z_{j,I}(t)-\gamma z_{i,I}(t) \\
    \frac{\dd}{\dd t}z_{i,R}(t)=& \gamma z_{i,I}(t),
\end{align*}
while BHMFA is
\begin{align*}
    \frac{\dd}{\dd t}x_{k,S}(t)=&-\beta x_{k,S}(t)\sum_{l}w_{kl}\pi_lx_{l,I}(t) \\
    \frac{\dd}{\dd t}x_{k,I}(t)=&\beta x_{k,S}(t)\sum_{l}w_{kl}\pi_lx_{l,I}(t)-\gamma x_{k,I}(t) \\
    \frac{\dd}{\dd t}x_{k,R}(t)=& \gamma x_{k,I}(t).
\end{align*}
\subsubsection*{The catalyst process}
Contrary to previous example, we introduce this simple stochastic process mainly for theoretical purposes. It has the advantage of being an integrable system while having enough complexity to retain some information about the structure of the underlying graph.

The state space is $S=\{a,b,c \}$ from which we call state $c$ the catalyst. There is only one kind of transition where a vertex in the catalyst state can make one of its neighbors jump from state $a$ to $b$ ate rate $\frac{1}{N \rho},$ that is $q_{c,a \to b}=1.$

The corresponding ODEs are
\begin{align*}
    \frac{\dd }{\dd t}z_{i,a}(t)=&-\frac{1}{N \rho}z_{i,a}(t) \sum_{j}a_{ij}z_{j,c}(t)  \\
    \frac{\dd }{\dd t}z_{i,b}(t)=&\frac{1}{N \rho}z_{i,a}(t) \sum_{j}a_{ij}z_{j,c}(t) \\
    \frac{\dd}{\dd t}z_{i,c}(t)=&0,
\end{align*}
and
\begin{align*}
    \frac{\dd}{\dd t}x_{k,a}(t)=&-x_{k,a}(t) \sum_{l}w_{kl} \pi_l x_{l,c}(t) \\
    \frac{\dd}{\dd t}x_{k,b}(t)=&x_{k,a}(t) \sum_{l}w_{kl} \pi_l x_{l,c}(t) \\
    \frac{\dd}{\dd t}x_{k,c}(t)=&0.
\end{align*}

It is easy to see that the catalyst variables remain constants and state $a$ evolves like
\begin{align}
\label{eq:catalyst}
\begin{split}
    z_{i,a}(t)=&z_{i,a}(0) \exp \left(- \left( \frac{1}{N \rho} \sum_{j}a_{ij}z_{j,c}(0) \right)t \right) \\
    x_{k,a}(t)=&x_{k,a}(0) \exp \left(-\left(\sum_{l}w_{kl} \pi_l x_{l,c}(0) \right)t \right).
\end{split}
\end{align}

\subsubsection*{The degree process}
The degree process is even simpler than the catalyst process. Again we introduce this process for theoretical purposes. The state space is $S=\{a,b \}.$ Any vertex, regardless their state  can make its neighbors jump from state $a$ to $b$ at rate $\frac{1}{N \rho}$ making $q_{a,a \to b}=q_{b,a \to b}=1.$ In total, a vertex in state $a$ transitions at rate $\frac{1}{N \rho} \delta_i.$

Since the transition rate of vertex $i$ depends only on its own degree, the evolution of the vertices are conditionally independent with respect to $G$. It is easy to check that
\begin{align}
\label{eq:degree_process}
\pr_{G}(\xi_{i,a}(t)=1)=\E_G \left(\xi_{i,a}(t) \right)=z_{i,a}(t)=&z_{i,a}(0)e^{-\frac{\delta_i}{N \rho}t} \\
x_{k,a}(t)=&\bar{z}_{k,a}(0)e^{-t}.
\end{align}

\section{Results}\label{sec:results}

For the rest of the paper the constants of $O(\cdot), \Omega(\cdot) $ and $\Theta(\cdot)$ may depend on $t,q_{s' \to s}, q_{\tilde{s};s' \to s}$ and other fixed constants.
Recall $x_{k,s}(t)$ and $\bar{\xi}_{k,s}(t)$ from \eqref{eq:x} and \eqref{eq:xi_bar} respectively.

Firstly, we deal with general, conditionally independent, $G$ measurable initial conditions.

\begin{theorem}
\label{t:general_upper_bound}
For the general class of processes introduced in \ref{sec:stoch} we have:
\begin{align}
\label{eq:t:general_upper_bound}
   \sup_{0 \leq \tau \leq t} \max_{\substack{1 \leq k \leq K \\ s \in \mathcal{S}}} \|\bar{\xi}_{k,s}(\tau)-x_{k,s}(\tau) \|=O \left( \frac{1}{\sqrt{d}} \right).  
\end{align}
\end{theorem}

If we do not assume anything about the initial conditions, then the upper bound \eqref{eq:t:general_upper_bound} can be achieved as a lower bound for the catalyst process on a single community graph, i.e.  on an Erdős-Rényi graph, making it in some sense tight. 
\begin{theorem}
 \label{t:general_lower_bound}
 For the catalyst process introduced in \ref{subsec:examples} on an Erdős-Rényi graph with density $\rho<1-c$, for some $c>0$, we can choose  conditionally independent, $G$ measurable initial conditions such that
 \begin{align}
     \|\bar{\xi}_{1,a}(1)-x_{1,a}(1) \|=\Theta \left( \frac{1}{\sqrt{d}} \right).
 \end{align}
\end{theorem}
 Secondly, we show how the error term changes when we restrict ourself to conditionally independent, $G$ measurable initial conditions that are homogeneous enough.
 \begin{theorem}
   \label{t:general_homogeneous_upper_bound}
   Let $\mathcal{E}_0$ denote the event
   \begin{align}\label{eq:hom_condition}
   \mathcal{E}_0= \left\{\forall s \in \mathcal{S}, \ \frac{1}{N}\sum_{k} \sum_{i \in I_k} \left(z_{i,s}(0)-\bar{z}_{k,s}(0) \right)^2 \leq \frac{c_0}{d}\right \}.
\end{align}
Assume $\pr \left( \mathcal{E}_0^c \right)$ decays superpolinomially, that is for any $r \in \mathbb{Z}^{+}$ one has $N^r \pr \left( \mathcal{E}_0^c \right) \to 0$ as $N \to \infty.$ Then for any process
\begin{align}
  \label{eq:general_homogeneous_upper_bound}
  \sup_{0 \leq \tau \leq t} \max_{\substack{1 \leq k \leq K \\ s \in \mathcal{S}}} \|\bar{\xi}_{k,s}(\tau)-x_{k,s}(\tau) \|=O \left( \max \left\{\frac{1}{d}, \frac{1}{\sqrt{N}} \right \}\right).  
\end{align}
 \end{theorem}

Lastly, we show that for homogeneous initial conditions the bound \eqref{eq:general_homogeneous_upper_bound} is also tight.

\begin{theorem}
 \label{t:homogeneous_lower_bound}
 Take an Erdős-Rényi graph with density $\rho$ and the degree process ( introduced in \ref{subsec:examples}) with initial conditions $z_{i,a}(0)=1$ for all $i \in V$ (making $\pr \left(\mathcal{E}_0^c \right)=0$). Then
 \begin{align}
     \|\bar{\xi}_{1,a}(1)-x_{1,a}(1) \|= \Theta \left( \max \left\{\frac{1}{d}, \frac{1}{\sqrt{N}} \right \} \right).
 \end{align}
\end{theorem}

\subsection{Discussion of the results}\label{subsec:discussion}
The above results lead to some interesting interpretations. One such could be phrased as ``\emph{how much structural information (local or global) does an antagonist need to mess up the initial conditions in order to produce the worst case scenario approximation?}''. Our theorems give an answer to this question. Indeed is true that Theorems \ref{t:general_upper_bound} and \ref{t:general_lower_bound} give a framework for general initial conditions, but we can also acknowledge that condition \eqref{eq:hom_condition} defines a quite natural family of initial conditions. Suppose that the antagonist has to explore the network in order to decide how to choose the initial condition at each site. This means that at each step he can discover only partially the network and decide the initial conditions based only on past information or the new one he can gather from a local neighborhood.  We could therefore model the antagonist as a random walk and is easy to see that the sets of initial condition he could come up with respect condition \eqref{eq:hom_condition}.

To elaborate on this, look at at the SI process, a version of the SIR process where the recover rate is set at $\gamma=0$ (and by simple rescaling of the time, we may also set $\beta=1$), so the state space is reduces to $\mathcal{S}=\{S,I \}.$  For the sake of simplicity, further assume that the underlying network is an Erdős-Rényi graph with parameter $\rho$ such that $d \leq N^{\frac{1}{2}-\varepsilon}$ for some $0<\varepsilon <\frac{1}{2}$ for large enough $N$.

As the degrees are given by local information, we may set the initial conditions to be $1-z_{i,S}(0)=z_{i,I}(0)=\kappa \frac{\delta_i}{d}$ with some suitably small $\kappa$. This makes the error term in \eqref{eq:hom_condition} to be proportional to
\begin{align*}
 \frac{1}{N} \sum_{i=1}^N \left | \frac{\delta_i}{d}-\frac{\bar{\delta}}{d} \right |^2 
\end{align*} 
where $\bar{\delta}:=\frac{1}{N} \sum_{i =1}^N \delta_i$ is the empirical degree average.

Notice that
$$\left | \frac{\delta_i}{d}-\frac{\bar{\delta}}{d} \right |^2 \leq 2 \left(\left | \frac{\delta_i}{d}-1 \right |^2+\left | \frac{\bar{\delta}}{d}-1 \right |^2\right).$$
Furthermore, we get that $\pr \left(\left|\frac{\bar{\delta}}{d}-1\right|^2 \geq \frac{1}{d} \right) \leq 2e^{-c N} $ with some $c>0$ when $N$ is large enough. Therefore, the relevant term left is
\begin{align*}
 \zeta:=\frac{1}{N}\sum_{k=1}^K \sum_{i \in I_K} \left | \frac{\delta_i}{d}-1 \right |^2. 
\end{align*}

It is easy to see that $\E(\zeta) \leq \frac{1}{d}.$ Furthermore, from \cite[Lemma 6.5]{erdosSpectralStatisticsErdos2013} (with $k=2$) we obtain
$$\left |\zeta-\E(\zeta)  \right |=O \left( \frac{\left(\log N \right)^{2 \xi} }{\sqrt{N}}\right) \ll \frac{1}{d}$$
with superpolynomial decay where $\xi>0$ is a constant, satisfying condition \eqref{eq:hom_condition}.

Another possible option for the SI process compatible with the scenario described above, is to set the initial conditions to be (or be close to) $z_{i,I}(0)= \kappa \frac{(v_1)_i}{\|v_1 \|_1}$ where $v_1$ is the Perron vector of the adjacency matrix $B$ corresponding to the largest eigenvalue $\lambda_1$ and $\kappa$ is a small number such that $\kappa^2$ is negligible. To see why, note that when $z_{i,I}(t)$ is small, we may linearize the dynamics into
\begin{align*}
 \frac{\dd}{\dd t} z_{i,I}(t) \approx& \left(B z_{\cdot, I}(t) \right)_i \\
 z_{i,I}(t)\approx &\sum_{m} e^{\lambda_m t} \langle v_m, z_{\cdot,I}(0) \rangle (v_{m})_i\approx e^{\lambda_1t}\langle v_1, z_{\cdot,I}(0) \rangle (v_{1})_i
\end{align*}
given that $t$ is suitably large for the leading eigenvalue to dominate but not too large for the linearization to break down. From Theorem 2.16 equation (2.33) in \cite{erdosSpectralStatisticsErdos2013} it is easy to see that condition \eqref{eq:hom_condition} is satisfied in this case too. Note that letting the SI dynamics evolve until the Perron vector starts to dominate, introduces some correlation between the state of the vertices even for fixed graphs, which violates the conditional independence assumption. Nonetheless, in  \cite[Lemma 4.3]{keligerConcentrationMeanField2024a} some propagation of chaos results are given, thus, we believe that this issue is merely a technical one.

\section{Proofs}\label{sec:proofs}
In this section we present the proof for the theorems presented in the previous section. In subsection \ref{subsec:upper} we prove the upper bounds for Theorem \ref{t:general_upper_bound} and \ref{t:general_homogeneous_upper_bound}, while in subsection \ref{subsec:lower}  we prove Theorem \ref{t:general_lower_bound} and \ref{t:homogeneous_lower_bound}.

\subsection{Upper bounds}\label{subsec:upper}
We start our discussion for the upper bounds with a useful lemma.

\begin{lemma}
\label{l:degree_bound}
There are constants $c_1',c_2'>0$ such that
\begin{align*}
\pr \left( \bigcup_{i=1}^{N} \left \{ \delta_i \geq c_1' d \right \} \right) \leq N e^{-c_2' N^{\alpha}}.
\end{align*}

\begin{proof} (Lemma \ref{l:degree_bound})

The centered variables $X_{j}^{(i)}:=a_{ij}-p_{ij}$ are independent and $|X_{j}^{(i)}| \leq 1.$ Note that
\begin{align*}
\sum_{j=1}^{N} \E \left( \left(X_{j}^{(i)} \right)^2 \right) =\sum_{j=1}^{N} p_{ij}(1-p_{ij}) \leq \sum_{j=1}^{N} p_{ij}=d_i.
\end{align*}
Furthermore,
\begin{align*}
(i \in I_k) \ d_i= \rho \sum_{l=1} w_{kl} N_l \leq M N \rho \overset{\eqref{eq:d}}{\leq} C d    
\end{align*}
for some constant $C>0$.

Set $c_1':=C+1$. Bernstein's inequality results in
\begin{align*}
&\pr \left(\delta_i \geq c_1' d \right) \leq \pr \left(\delta_i \geq d_i+d \right)=\pr \left(\sum_{j}X_j^{(i)} \geq d \right)\leq\\
& \exp \left(-\frac{\frac{1}{2}d^2}{\sum_{j=1}^{N} \E \left( \left(X_{j}^{(i)} \right)^2 \right)+\frac{1}{3}d} \right) \leq \exp \left(-\frac{\frac{1}{2}}{C+\frac{1}{3}}d \right)=:e^{-c_2'd} \leq e^{-c_2' N^{\alpha}}.
\end{align*}
\end{proof}
\end{lemma}

\begin{definition}
Recall Corollary \ref{c:spectral} and Lemma \ref{l:degree_bound}. Define the events $\mathcal{A},\mathcal{B}$ and $\mathcal{E}$ as
\begin{align}
\label{eq:events}
\begin{split}
 \mathcal{A}:=& \bigcup_{i=1}^{N} \{\delta_i \leq c_1' d \},  \\
 \mathcal{B}:=& \left \{\|B-\hat{B} \|_2 \leq \frac{c_1'}{\sqrt{d}} \right \}, \\
 \mathcal{E}:=&\mathcal{A} \cap \mathcal{B}.
\end{split}
\end{align}
Clearly
\begin{align*}
\pr \left(\mathcal{E}^{c} \right) \leq Ne^{-c_2' N^{\alpha}}+2e^{-c_2 N^{2\beta}}
\end{align*}
which decays super-polinomially.
\end{definition}

\begin{lemma}
\label{l:xi_bar_error_decompose}
 Define $\|X \|_{G}^2:=\E_G(X^2).$  Then on the event $\mathcal{E}$ for any $t \geq 0$
\begin{align*}
\sup_{0 \leq \tau \leq t}\left\|\bar{\xi}_{k,s}(\tau)-x_{k,s}(\tau) \right \|_{G} =&O \left(\max \left \{\frac{1}{d}, \frac{1}{\sqrt{N}} \right \} \right)\\
&+\sup_{0 \leq \tau \leq t}\left|\bar{z}_{k,s}(\tau)-x_{k,s}(\tau) \right|,
\end{align*}
where $O(\cdot)$ depend on $t$.
\end{lemma}

\begin{proof} (Lemma \ref{l:xi_bar_error_decompose})

\begin{align*}
\left \|\bar{\xi}_{k,s}(\tau)-x_{k,s}(\tau) \right\|_{G} \leq& \left\|\bar{\xi}_{k,s}(\tau)-\E_G \left(\bar{\xi}_{k,s}(\tau) \right) \right \|_{G}+\left|\E_G \left(\bar{\xi}_{k,s}(\tau) \right)-\bar{z}_{k,s}(\tau) \right| \\
&+\left|\bar{z}_{k,s}(\tau)-x_{k,s}(\tau) \right|
\end{align*}

Note that on $\mathcal{E}$
\begin{align*}
\sum_{j=1}^N b_{ij}=\frac{1}{N \rho}\sum_{j=1}^N a_{ij}=\frac{d_i}{N \rho} \leq \frac{c_1 d}{N \rho}=O(1)    
\end{align*}
uniformly in $i$, which means the total incoming rates are uniformly bounded in $i$.

Based on Theorem Theorem 4.4 in \cite{keligerConcentrationMeanField2024a}
\begin{align*}
\left \|\bar{\xi}_{k,s}(\tau)-\E_G \left(\bar{\xi}_{k,s}(\tau) \right) \right \|_{G}^2=\operatorname{Var}_{G}(\bar{\xi}_{k,s}(\tau))=O \left(\frac{1}{N_k} \right)=O \left(\frac{1}{N} \right).
\end{align*}

Also, the maximal interaction rate is $$\frac{1}{N \rho} \max_{\substack{s,s',\tilde{s} \\ s \neq s' }}q_{\tilde{s},s \to s'}=O \left( \frac{1}{d} \right),$$
therefore, by Theorem 4.14 in \cite{keligerConcentrationMeanField2024a} we obtain
\begin{align*}
\left |\E_G \left(\bar{\xi}_{k,s}(\tau) \right)-\bar{z}_{k,s}(\tau) \right| \leq \frac{1}{N_k}\sum_{i \in I_k}  \left| \pr_{G}\left(\xi_{i,s}(\tau)=1 \right)-z_{i,s}(\tau) \right |=O\left(\frac{1}{d} \right).
\end{align*}

All the bounds above are uniform in $0 \leq \tau \leq t.$
\end{proof}

\begin{definition}
$\hat{z}_{i,s}(t)$ is the solution of \eqref{eq:NIMFA} where the matrix $B$ is replaced by $\hat{B}$:
\begin{align}
\label{eq:z_hat}
\frac{\dd}{\dd t} \hat{z}_{i,s}(t)=\sum_{s' \in \mathcal{S}}q_{s' \to s}\hat{z}_{i,s'}(t)+\sum_{s', \tilde{s} \in \mathcal{S}}q_{\tilde{s};s' \to s}\hat{z}_{i,s'}(t) \left(\hat{B}\hat{z}_{\cdot,\tilde{s}}(t) \right)_{i}
\end{align}
with initial conditions $\hat{z}_{i,s}(0)=z_{i,s}(0)=\pr \left(\xi_{i,s}(0)=1 \right).$   
\end{definition}

\begin{remark}
\label{r:B_hat}
For an arbitrary vector $u \in \mathbb{R}^N$
\begin{align}
\label{eq:B_hat_identity}
(i \in I_k) \ \left(\hat{B}u \right)_i=\sum_{l=1}^K \sum_{j \in I_l} \hat{b}_{ij}u_j=\sum_{l=1}^{K} \frac{1}{N \rho} \rho w_{kl} \sum_{j \in I_l} u_j=\sum_{l=1}^{K} w_{kl} \pi_l \bar{u}_l
\end{align}
which depends only on $k$.

By applying \eqref{eq:B_hat_identity} and taking averages on the block $I_k$ on both sides of \eqref{eq:z_hat} it is easy to see that
\begin{align}
\label{eq:x_identity}
x_{k,s}(t)=\frac{1}{N_k} \sum_{i \in I_k} \hat{z}_{i,s}(t)
\end{align}
solves \eqref{eq:x}.
\end{remark}

\begin{definition}
\begin{align}
\label{eq:Delta}
\begin{split}
\Delta_{i,s}(t):=&z_{i,s}(t)-\hat{z}_{i,s}(t), \\
\tilde{\Delta}_{i}(t):=&\sup_{0 \leq \tau \leq t} \max_{s \in \mathcal{S}} \left| \Delta_{i,s}(t) \right|, \\
\tilde{\Delta}(t):=& \left(\tilde{\Delta}_{i}(t) \right)_{i=1}^N.
\end{split}
\end{align}
\end{definition}

\begin{lemma}
\label{l:Delta}
On the event $\mathcal{E}$  
\begin{align*}
\frac{1}{N}\|\tilde{\Delta}(t) \|_2^2=O \left( \frac{1}{d} \right),
\end{align*}
where $O(\cdot)$ depend on $t$.
\end{lemma}

\begin{proof} (Lemma \ref{l:Delta})

Recall \eqref{eq:NIMFA2}. After integration and subtracting the analogue with $\hat{B}$ we get
\begin{align*}
\Delta_{i,s}(t)=&\sum_{s' \in \mathcal{S}} q_{s' \to s} \int_{0}^{t} \Delta_{i,s'}(\tau) \dd \tau\\
&+ \int_{0}^{t} \sum_{s', \tilde{s} \in \mathcal{S}}q_{\tilde{s}; s' \to s}\left[z_{i,s'}(\tau) \left(B z_{\cdot, \tilde{s}}(\tau) \right)_{i}-\hat{z}_{i,s'}(\tau) \left(\hat{B} \hat{z}_{\cdot, \tilde{s}}(\tau) \right)_{i} \right]  \dd \tau 
\end{align*}

After applying Cauchy-Schwartz inequality for sums and integrals we obtain 
\begin{align}
\label{eq:delta^2}
\begin{split}
\Delta_{i,s}^2(t)\leq &C_1\left( \sum_{s' \in \mathcal{S}} \int_{0}^{t} \Delta_{i,s'}^2(\tau) \dd \tau \right.\\
&\left.+ \int_{0}^{t} \sum_{s', \tilde{s} \in \mathcal{S}}\left[z_{i,s'}(\tau) \left(B z_{\cdot, \tilde{s}}(\tau) \right)_{i}-\hat{z}_{i,s'}(\tau) \left(\hat{B} \hat{z}_{\cdot, \tilde{s}}(\tau) \right)_{i} \right]^2  \dd \tau  \right)
\end{split}
\end{align}
for some $C_1>0$ (depending on $t$).

Let us reformulate the expression between the square brackets.
\begin{align*}
&z_{i,s'}(\tau) \left(B z_{\cdot, \tilde{s}}(\tau) \right)_{i}-\hat{z}_{i,s'}(\tau) \left(\hat{B} \hat{z}_{\cdot, \tilde{s}}(\tau) \right)_{i}=\\
&\left(B z_{\cdot, \tilde{s}}(\tau) \right)_{i}\Delta_{i,s'}(\tau)+\hat{z}_{i,s'}(\tau) \left[\left((B-\hat{B})z_{\cdot,s'}(\tau) \right)_i+\left(\hat{B}\Delta_{\cdot,\tilde{s}}(\tau) \right)_i \right]
\end{align*}
On the even $\mathcal{E}$ we have $\left(B z_{\cdot, \tilde{s}}(\tau) \right)_{i}=O(1)$ (uniformly in $i$) making
\begin{align*}
&\left[z_{i,s'}(\tau) \left(B z_{\cdot, \tilde{s}}(\tau) \right)_{i}-\hat{z}_{i,s'}(\tau) \left(\hat{B} \hat{z}_{\cdot, \tilde{s}}(\tau) \right)_{i} \right]^2=\\
&O \left(\Delta_{i,s'}^2(\tau)+\left((B-\hat{B})z_{\cdot,s'}(\tau) \right)_i^2+ \left(\hat{B}\Delta_{\cdot,\tilde{s}}(\tau) \right)_i^2\right).
\end{align*}

Furthermore, also on the event $\mathcal{E}$
\begin{align*}
\frac{1}{N}\sum_{i=1}^N\left((B-\hat{B})z_{\cdot,s'}(\tau) \right)_i^2 =& \frac{1}{N}\left \|(B-\hat{B})z_{\cdot,s'}(\tau) \right \|_2^2 \\
\leq& \frac{1}{N} \|B-\hat{B} \|_2^2 \|z_{\cdot, s'}(\tau) \|_2^2 \leq \|B-\hat{B} \|_2^2=O \left(\frac{1}{d} \right), \\
\frac{1}{N}\sum_{i=1}^N  \left(\hat{B}\Delta_{\cdot,\tilde{s}}(\tau) \right)_i^2=& \frac{1}{N}\left\| \hat{B} \Delta_{\cdot, \tilde{s}}(\tau) \right\|_2^2\leq \frac{\|\hat{B} \|_2^2}{N} \left\| \Delta_{\cdot, \tilde{s}}(\tau) \right\|_2^2 \\
=&O \left(\frac{1}{N} \left\| \Delta_{\cdot, \tilde{s}}(\tau) \right\|_2^2 \right).
\end{align*}

Taking maximums in $s,\tilde{s},s' $ taking supreme in $\tau \in [0,t]$ and taking averages in $i$ and applying the bounds above results in
\begin{align*}
\frac{1}{N} \| \tilde{\Delta}(t) \|_2^2 \leq C_2 \int_{0}^{t}\frac{1}{N} \| \tilde{\Delta}(\tau) \|_2^2 \dd \tau +O\left(\frac{1}{d} \right)
\end{align*}
for some $C_2>0$. Applying Grönwall's inequality finishes the proof.
\end{proof}

These preparations are now enough to prove Theorem \ref{t:general_upper_bound}.

\begin{proof}(Theorem \ref{t:general_upper_bound})
 
 Recall
 \begin{align*}
 \pr \left( \mathcal{E}^{c}\right) \leq Ne^{-c_2'N^{\alpha}}+2e^{-c_2N^{2\beta}}  \ll \frac{1}{d}.
 \end{align*}

 Since $\bar{\xi}_{k,s}(t), x_{k,s}(t)$ are in $[0,1]$ we have
 \begin{align*}
  \|\bar{\xi}_{k,s}(\tau)- x_{k,s}(\tau) \|^2 \leq\E \left(\|\bar{\xi}_{k,s}(\tau)- x_{k,s}(\tau) \|_{G}^2 \1{\mathcal{E}} \right)+\pr \left(\mathcal{E}^{c} \right).   
 \end{align*}
Using the identity $(a+b)^2 \leq 2(a^2+b^2)$ \eqref{eq:x_identity} and Lemma \ref{l:xi_bar_error_decompose} on the event $\mathcal{E}$ we get
\begin{align*}
\sup_{0 \leq \tau \leq t}\|\bar{\xi}_{k,s}(\tau)- x_{k,s}(\tau) \|_{G}^2 \leq& 2 \sup_{0 \leq \tau \leq t}|\bar{z}_{k,s}(\tau)- x_{k,s}(\tau) |^2+O\left(\frac{1}{d} \right)\\
 \leq & 2\left(\frac{1}{N_k}\sum_{i \in I_K} \tilde{\Delta}_{i}(t) \right)^2+O \left(\frac{1}{d} \right) \\
\leq & \frac{2}{c_{block}^2} \frac{1}{N}\|\tilde{\Delta}(t) \|_2^2+O\left(\frac{1}{d} \right)=O\left(\frac{1}{d} \right).
\end{align*}
\end{proof}

Before we can prove Theorem \ref{t:general_homogeneous_upper_bound} we need a few more extra steps.

\begin{definition}
Let $\tilde{z}_{i,s}(t)$ be the solution of \eqref{eq:z_hat} with initial condition $\tilde{z}_{i,s}(0)=\bar{z}_{k,s}(0)=x_{k,s}(0)$ for all $k \in \{1, \dots, K\}, \ i \in I_k $ and $s \in \mathcal{S}. $ It is easy to see that for all $t \geq 0$ we also have
\begin{align}
\label{eq:x_identity2}
    \tilde{z}_{i,s}(t)=x_{k,s}(t) \ (i \in I_k)
\end{align}
by checking that it solves \eqref{eq:z_hat} (and it trivially satisfies the initial conditions).
\end{definition}

\begin{definition}
\begin{align}
\label{eq:Delta*}
\begin{split}
\Delta_{i,s}^{*}(t):=&\hat{z}_{i,s}(t)-\tilde{z}_{i,s}(t), \\
\tilde{\Delta}_{i}^{*}(t):=&\sup_{0 \leq \tau \leq t} \max_{s \in \mathcal{S}} \left| \Delta_{i,s}^{*}(t) \right|, \\
\tilde{\Delta}^{*}(t):=& \left(\tilde{\Delta}_{i}^{*}(t) \right)_{i=1}^N.
\end{split}
\end{align}
\end{definition}

\begin{lemma}
\label{l:Delta*}
\begin{align*}
\frac{1}{N} \left\| \tilde{\Delta}^*(t) \right\|_2^2=O \left(\frac{1}{N} \left\| \tilde{\Delta}^*(0) \right \|_2^2 \right)
\end{align*}
\end{lemma}

\begin{proof}(Lemma \ref{l:Delta*})

Similarly to the proof of Lemma \ref{l:Delta}:
\begin{align*}
\Delta_{i,s}^*(t)=&\Delta_{i,s}^*(0)+\sum_{s' \in \mathcal{S}} q_{s' \to s} \int_{0}^{t} \Delta_{i,s'}^*(\tau) \dd \tau\\
&+ \int_{0}^{t} \sum_{s', \tilde{s} \in \mathcal{S}}q_{\tilde{s}; s' \to s}\left[\hat{z}_{i,s'}(\tau) \left(\hat{B} \hat{z}_{\cdot, \tilde{s}}(\tau) \right)_{i}-\tilde{z}_{i,s'}(\tau) \left(\hat{B} \tilde{z}_{\cdot, \tilde{s}}(\tau) \right)_{i} \right]  \dd \tau \\
\left(\Delta_{i,s}^*(t)\right)^2\leq &C_1\left(\left(\Delta_{i,s}^*(0) \right)^2 +\sum_{s' \in \mathcal{S}} \int_{0}^{t} \left(\Delta_{i,s'}^*(\tau) \right)^2 \dd \tau \right.\\
&\left.+ \int_{0}^{t} \sum_{s', \tilde{s} \in \mathcal{S}}\left[\hat{z}_{i,s'}(\tau) \left(\hat{B} \hat{z}_{\cdot, \tilde{s}}(\tau) \right)_{i}-\tilde{z}_{i,s'}(\tau) \left(\hat{B} \tilde{z}_{\cdot, \tilde{s}}(\tau) \right)_{i} \right]^2  \dd \tau  \right).
\end{align*}

For the expression in square brackets, note that $\left(\hat{B} \hat{z}_{\cdot, \tilde{s}}(\tau) \right)_{i}=O(1)$ uniformly in $i$ according to Remark \ref{r:B_hat}.

\begin{align*}
\hat{z}_{i,s'}(\tau) \left(\hat{B} \hat{z}_{\cdot, \tilde{s}}(\tau) \right)_{i}-\tilde{z}_{i,s'}(\tau) \left(\hat{B} \tilde{z}_{\cdot, \tilde{s}}(\tau) \right)_{i}=&\left(\hat{B} \hat{z}_{\cdot, \tilde{s}}(\tau) \right)_{i} \Delta_{i,s'}^{*}(\tau)+\tilde{z}_{i,s'}(\tau) \left(\hat{B} \Delta_{\cdot, \tilde{s}}^*(\tau) \right)_{i} \\
\left[\hat{z}_{i,s'}(\tau) \left(\hat{B} \hat{z}_{\cdot, \tilde{s}}(\tau) \right)_{i}-\tilde{z}_{i,s'}(\tau) \left(\hat{B} \tilde{z}_{\cdot, \tilde{s}}(\tau) \right)_{i} \right]^2 \leq &2 \left(\left(\Delta_{i,s'}^{*}(\tau) \right)^2+\left(\hat{B} \Delta_{\cdot, \tilde{s}}^*(\tau) \right)_{i}^2 \right)
\end{align*}

Clearly, $\|\hat{B} \|_2=O(1)$ making
\begin{align*}
\frac{1}{N} \left \| \tilde{\Delta}^*(t) \right \|_2^2 \leq& C_2 \left(\frac{1}{N} \left \| \tilde{\Delta}^*(0) \right \|_2^2+\int_{0}^{t} \frac{1}{N} \left \| \tilde{\Delta}^*(\tau) \right \|_2^2 \dd \tau\right) \\
\frac{1}{N} \left \| \tilde{\Delta}^*(t) \right \|_2^2 =& O \left(\frac{1}{N} \left \| \tilde{\Delta}^*(0) \right \|_2^2 \right)
\end{align*}
by Grönwall's inequality.
\end{proof}

\begin{lemma}
\label{l:homogenous_trick}
Recall $\mathcal{E}_0$ from Theorem \ref{t:general_homogeneous_upper_bound}. Note that the event can be reexpressed as  
\begin{align*}
    \mathcal{E}_0 = \left \{ \frac{1}{N} \left \| \tilde{\Delta}^*(0) \right \|_2^2 \leq\frac{c_0}{d} \right \}.
\end{align*}
 Assume $\pr \left(\mathcal{E}_0^c \right)$ decays super polinomially in $N$. Then
\begin{align*}
\sup_{0 \leq \tau \leq t} \max_{\substack{1 \leq k \leq K \\S s\in \mathcal{S}}}\left \|\bar{z}_{k,s}(\tau)-x_{k,s}(\tau) \right \|=O \left( \frac{1}{d} \right).
\end{align*}
\end{lemma}

\begin{proof} (Lemma \ref{l:homogenous_trick})

Define the event $\bar{\mathcal{E}}:=\mathcal{E} \cap \mathcal{E}_0$. Also, let $h_{k,s}(t):=\bar{z}_{k,s}(t)-x_{k,s}(t).$ Note that $h_{k,s}(0)=0.$

\begin{align*}
h_{k,s}(t)=&\sum_{s' \in \mathcal{S}}q_{s' \to s} \int_{0}^{t} h_{k,s'}(\tau) \dd \tau\\
&+\sum_{s', \tilde{s} \in \mathcal{S}}q_{\tilde{s}, s' \to s} \int_{0}^t \frac{1}{N_k}\sum_{i \in I_K} \left[z_{i,s'}(\tau) \left(B z_{\cdot, \tilde{s}}(\tau) \right)_i-\tilde{z}_{i,s'}(\tau) \left(\hat{B} \tilde{z}_{\cdot, \tilde{s}}(\tau) \right)_i \right ] \dd \tau
\end{align*}
We split the terms in the square bracket into two parts:
\begin{align}
\label{eq:homogeneous_trick_split}
\begin{split}
    &z_{i,s'}(\tau) \left(B z_{\cdot, \tilde{s}}(\tau) \right)_i-\tilde{z}_{i,s'}(t) \left(\hat{B} \tilde{z}_{\cdot, \tilde{s}}(\tau) \right)_i =\\
    &z_{i,s'}(\tau) \left(\left(B-\hat{B}\right) z_{\cdot, \tilde{s}}(\tau) \right)_i+z_{i,s'}(\tau) \left(\hat{B} z_{\cdot, \tilde{s}}(\tau) \right)_i-\tilde{z}_{i,s'}(\tau) \left(\hat{B} \tilde{z}_{\cdot, \tilde{s}}(\tau) \right)_i.
\end{split}
\end{align}
We start with the first term.

Define $\Delta_{i,s}^{**}(t):=\Delta_{i,s}(t)+\Delta_{i,s}^{*}(t)$. From Lemma \ref{l:Delta} and \ref{l:Delta*} we know that on the event $\bar{\mathcal{E}}$ $\frac{1}{N} \left \| \Delta_{\cdot ,s}^{**}(t)\right \|_2^2=O \left(\frac{1}{d} \right)$ uniformly in $ 0 \leq \tau \leq t.$ 

Clearly, $z_{i,s}(\tau)=\tilde{z}_{i,s}(\tau)+\Delta_{i,s}^{**}(\tau),$ making
\begin{align*}
\frac{1}{N_k}\sum_{i \in I_k}z_{i,s'}(\tau) \left(\left(B-\hat{B}\right) z_{\cdot, \tilde{s}}(\tau) \right)_i=&\frac{1}{N_k}\sum_{i \in I_k}\tilde{z}_{i,s'}(\tau) \left(\left(B-\hat{B}\right) \tilde{z}_{\cdot, \tilde{s}}(\tau) \right)_i\\
&+O \left( \max_{s \in \mathcal{S}}\frac{1}{N}\sum_{i } \left |\left(\left(B-\hat{B} \right)\Delta_{\cdot, s}^{**}(\tau) \right)_i \right|  \right).
\end{align*}
 Since we are on  the event $\bar{\mathcal{E}}$
 \begin{align*}
    \frac{1}{N}\sum_{i } \left |\left(\left(B-\hat{B} \right)\Delta_{\cdot, s}^{**}(\tau) \right)_i \right|  \leq& \sqrt{\frac{1}{N} \left\|(B-\hat{B}) \Delta_{\cdot, s}^{**}(\tau) \right\|_2^2} \\
    \leq& \sqrt{ \underbrace{\left\|(B-\hat{B})  \right\|_2^2}_{=O \left( \frac{1}{d}\right)} \underbrace{\frac{1}{N} \left \|\Delta_{\cdot, s}^{**}(\tau) \right \|_2^2}_{=O \left(\frac{1}{d} \right)}} =O \left(\frac{1}{d} \right).
 \end{align*}
For the other term note that  for $i\in I_k$ $\tilde{z}_{i,s}(\tau)=x_{k,s}(\tau)$ making
\begin{align*}
\eta_{k,s', \tilde{s}}(\tau):=&\frac{1}{N_k}\sum_{i \in I_k}\tilde{z}_{i,s'}(\tau) \left(\left(B-\hat{B}\right) \tilde{z}_{\cdot, \tilde{s}}(\tau) \right)_i\\
=& \sum_{l}x_{k,s'}(\tau)x_{l,\tilde{s}}(\tau) \frac{1}{N_k}\sum_{\substack{i \in I_k \\ j \in I_l}} (b_{ij}-\hat{b}_{ij}).
\end{align*}
Since $\E(b_{ij})=\hat{b}_{ij}$  (except for $i=j$) and $b_{ij}=\frac{1}{N \rho}a_{ij}$ are independent  we have
\begin{align*}
\left \| \eta_{k,s', \tilde{s}}(\tau) \right \|=O \left(\sqrt{\frac{1}{\left(N_k N \rho\right)^2}\sum_{\substack{i \in I_k \\ j \in I_l}} p_{ij}(1-p_{ij})}+ \frac{1}{N} \right)=O \left(\frac{1}{\sqrt{Nd}} \right)=O \left( \frac{1}{d} \right).
\end{align*}

We now turn to the other terms in \eqref{eq:homogeneous_trick_split}. Using \eqref{eq:B_hat_identity} we know that $(\hat{B}u)_i$ only depend on $k$ where $i \in I_k$ for any $u \in \mathbb{R}^N.$

Fix any $i_0 \in I_k.$
\begin{align*}
&\frac{1}{N_k}\sum_{i \in I_k} z_{i,s'}(\tau) \left(\hat{B} z_{\cdot, \tilde{s}}(\tau) \right)_i-\tilde{z}_{i,s'}(\tau) \left(\hat{B} \tilde{z}_{\cdot, \tilde{s}}(\tau) \right)_i= \\
&\bar{z}_{k,s'}(\tau) \left(\hat{B} z_{\cdot, \tilde{s}}(\tau) \right)_{i_0}-x_{k,s'}(\tau) \left(\hat{B} \tilde{z}_{\cdot, \tilde{s}}(\tau) \right)_{i_0}\\
=&\left(\hat{B} z_{\cdot, \tilde{s}}(\tau) \right)_{i_0}h_{k,s'}(\tau)+x_{k,s}(\tau)\left(\hat{B} \left(z_{\cdot,\tilde{s}}(\tau)-\tilde{z}_{\cdot, \tilde{s}}(\tau) \right) \right)_{i_0} \\
=&\underbrace{ \left(\hat{B} z_{\cdot, \tilde{s}}(\tau) \right)_{i_0}}_{=O(1)}h_{k,s'}(\tau)+x_{k,s}(\tau) \sum_{l}w_{kl}\pi_l h_{l,\tilde{s}}(\tau).
\end{align*}

Define
\begin{align*}
\bar{h}(\tau):=& \max_{\substack{1 \leq k \leq K \\ s\in \mathcal{S}}} \left\|h_{k,s}(\tau) \right \|, \\
\bar{h}^*(\tau):=&\max_{\substack{1 \leq k \leq K \\ s\in \mathcal{S}}} \left\|h_{k,s}(\tau) \1{\bar{\mathcal{E}}} \right \|.
\end{align*}
Since $\left |h_{k,s}(\tau)\right| \leq 1$ and $\pr \left( \bar{\mathcal{E}}^c\right)$ decays superpolynomially, we have $\bar{h}(\tau)=\bar{h}^*(\tau)+O \left( \frac{1}{d}\right).$

From the bounds above we got that there must be some $C>0$ such that
\begin{align*}
    \bar{h}^*(t) \leq O \left( \frac{1}{d} \right)+C \int_0^t \bar{h}^*(\tau) \dd \tau,
\end{align*}
resulting in 
$$\sup_{0 \leq \tau \leq t}\bar{h}(\tau)=\sup_{0 \leq \tau \leq t}\bar{h}^*(\tau)+O \left( \frac{1}{d} \right)=O \left( \frac{1}{d} \right).$$
\end{proof}

This leads us directly to the proof of Theorem \ref{t:general_homogeneous_upper_bound}.

\begin{proof}( Theorem \ref{t:general_homogeneous_upper_bound})
\begin{align*}
&\sup_{0 \leq \tau \leq t} \max_{ \substack{1 \leq k \leq K \\ s \in \mathcal{S}}}  \left\|\bar{\xi}_{k,s}(\tau)-x_{k,s}(\tau)  \right\| =\\
&\sup_{0 \leq \tau \leq t} \max_{ \substack{1 \leq k \leq K \\ s \in \mathcal{S}}}  \left\| \left(\bar{\xi}_{k,s}(\tau)-x_{k,s}(\tau) \right)\1{\mathcal{E}}  \right\|  +O \left(\max \left \{\frac{1}{d}, \frac{1}{\sqrt{N}} \right \} \right) \overset{ \textit{Lemma \ref{l:xi_bar_error_decompose} }}{\leq} \\
&\sup_{0 \leq \tau \leq t} \max_{ \substack{1 \leq k \leq K \\ s \in \mathcal{S}}}  \left\| \bar{z}_{k,s}(\tau)-x_{k,s}(\tau)  \right\|  +O \left(\max \left \{\frac{1}{d}, \frac{1}{\sqrt{N}} \right \} \right) \overset{ \textit{Lemma \ref{l:homogenous_trick} }}{=} \\
&O \left(\max \left \{\frac{1}{d}, \frac{1}{\sqrt{N}} \right \} \right)
\end{align*}

\end{proof}

\subsection{Lower bounds}\label{subsec:lower}

Now we turn our attention to the lower bounds. We set $K=1$ and $w_{11}=1$, that is, we take $G$ to be an Erdős-Rényi graph of density $\rho$. Clrearly, $d=(N-1) \rho$ for the Erdős-Rényi graph.

For the ease of notation we omit the $k$ indices in this section as there is only one block.

We start with Theorem \ref{t:general_lower_bound}.

\begin{lemma}
\label{l:lower_bound}
For the catalyst process on the Erdős-Rényi graph we have
\begin{align*}
\left|\bar{z}_a(1)-x_a(1) \right| \geq e^{-\bar{{z}}_{c}(0)} \left( \frac{1}{N} \left(z_{\cdot, a}(0) \right)^{\intercal}(B-\hat{B})z_{\cdot,c}(0) \right)-\frac{1}{2} \|B-\hat{B} \|_2^2. 
\end{align*}
\end{lemma}

\begin{proof}
Recall \eqref{eq:catalyst}. Define
\begin{align*}
\phi_i:=&\frac{1}{N \rho}\sum_{j}a_{ij}z_{j,c}(0)=\left(Bz_{\cdot, c}(0) \right)_i \\
\bar{\phi}_i:=& \left(\hat{B}z_{\cdot, c}(0) \right)_i=\bar{z}_{c}(0)=x_c(0).
\end{align*}
Clearly,
\begin{align*}
    \bar{z}_{a}(t)=&\frac{1}{N}\sum_{i}z_{i,a}(t)=\frac{1}{N}\sum_{i}z_{i,a}(0)e^{-\phi_i t} \\
    x_a(t)=&x_a(0)e^{-x_c(0)t}=\frac{1}{N}\sum_{i}z_{i,a}(0)e^{-\bar{\phi}_i t},
\end{align*}
making
\begin{align*}
\bar{z}_a(1)-x_a(1)=\frac{1}{N}\sum_{i} \left(e^{-\phi_i}-e^{-\tilde{\phi}_i} \right)z_{i,a}(0).
\end{align*}

Using the mean value theorem there is a $\xi_{i} \in [\tilde{\phi}_i, \phi_i]$ such that
\begin{align*}
 e^{-\phi_i}=e^{-\tilde{\phi}_i}-e^{-\tilde{\phi}_i} \left(\phi_i-\tilde{\phi}_i \right)+\frac{1}{2}e^{-\xi_i} \left(\phi_i-\tilde{\phi}_i \right)^2,   
\end{align*}
implying that
\begin{align*}
-\left(\bar{z}_a(1)-x_a(1) \right)=\frac{1}{N}\sum_{i}e^{-\bar{\phi}_i} \left(\phi_i-\tilde{\phi}_i \right)z_{i,a}(0)-\frac{1}{2N} \sum_{i}e^{-\xi_i}\left(\phi_i-\tilde{\phi}_i \right)^2 z_{i,a}(0).
\end{align*}

Observe that $\tilde{\phi}_{i}=\bar{z}_c(0)$ is independent of $i$.
\begin{align*}
&\frac{1}{N}\sum_{i}e^{-\tilde{\phi}_i} \left(\phi_i-\tilde{\phi}_i \right)z_{i,a}(0)=\frac{e^{-\bar{z}_a(0)}}{N}\sum_{i} \left(\phi_i-\tilde{\phi}_i \right)z_{i,a}(0)= \\
&\frac{e^{-\bar{z}_a(0)}}{N}\sum_{i} \left( \left( B-\hat{B}\right)z_{\cdot, c}(0) \right)_i z_{i,a}(0)=e^{-\bar{z}_a(0)} \left( \frac{1}{N}\left( z_{\cdot,a}(0)\right)^{\intercal}(B-\hat{B})z_{\cdot,c}(0) \right)
\end{align*}

Using $\xi_i \geq 0$ we also have that
\begin{align*}
&\frac{1}{2N} \sum_{i}e^{-\xi_i}\left(\phi_i-\tilde{\phi}_i \right)^2 z_{i,a}(0) \leq \frac{1}{2N} \sum_{i}\left(\phi_i-\tilde{\phi}_i \right)^2  \leq \\
& \frac{1}{2N} \left \| \phi-\tilde{\phi} \right\|_2^2=\frac{1}{2N} \left \| \left(B-\hat{B} \right)z_{\cdot,c}(0) \right\|_2^2 \leq  \\
& \frac{1}{2N} \left \| B-\hat{B}  \right\|_2^2 \left \| z_{\cdot,c}(0) \right\|_2^2 \leq \frac{1}{2} \left \| B-\hat{B}  \right\|_2^2
\end{align*}
concluding the proof.
\end{proof}

\begin{remark}
\label{r:Erdos}
Set the initial conditions such that $z_{i,a}(0)=z_{i,c}(0)=\frac{1}{2}u_i$ and $z_{i,b}(0)=1-u_i$ where $0 \leq u_i \leq 1.$ Furthermore, define $\bar{u}:=\frac{1}{N} \sum_{i} u_i.$

According to Lemma \ref{l:lower_bound} we get the inequality
\begin{align*}
\left|\bar{z}_a(1)-x_a(1) \right| \geq \frac{1}{4}e^{-\frac{1}{2}\bar{u}} \left( \frac{1}{N} u^{\intercal}(B-\hat{B})u \right)-\frac{1}{2} \|B-\hat{B} \|_2^2. 
\end{align*}

As we have seen before: $\|B-\hat{B} \|_2^2=O \left( \frac{1}{d} \right)$ which will be negligible compared to the main term.

The first quantity on the right is closly related to the \emph{modularity} of the Erdős-Rényi graph. Choose the vector $u$ to be the indicator of a set $H$: $u=\1{H}.$ Then

\begin{align*}
\frac{1}{4}e^{-\frac{1}{2}\bar{u}} \left( \frac{1}{N} u^{\intercal}(B-\hat{B})u \right)=&\frac{1}{4}e^{-\frac{1}{2}\frac{|H|}{N}} \left[\frac{1}{Nd} \left(e(H)-\frac{d}{N} |H |^2 \right) \right]-\frac{1}{N} \left(\frac{1}{4}e^{-\frac{1}{2}\frac{|H|}{N}} \frac{e(H)}{Nd} \right), 
\end{align*}
where $e(H)=\sum_{i,j \in H}a_{ij}$ and the last term comes $\frac{d}{N \rho}=1-\frac{1}{N}.$ Note that
\begin{align*}
\frac{1}{N} \left(\frac{1}{4}e^{-\frac{1}{2}\frac{|H|}{N}} \frac{e(H)}{Nd} \right) \leq \frac{1}{N} \frac{\bar{\delta}}{d}
\end{align*}
which has expectation $\frac{1}{N}$.

Using (article of erdos) or (corollary 2 page 4 Scott Bollobas) we have that the quantity 
\[
\left|\left(e(H)-\frac{d}{N} |H |^2 \right)\right| \geq \sqrt{p(1-p)}\frac{N^{3/2}}{80}
\]
  for some $H$ where $p$ stand for the \emph{empirical} edge density which strongly concentrates on $\rho, $ therefore, we have that the right hand side is $\Omega\left(\frac{1}{\sqrt{d}}\right)$ when $\rho\leq 1-c$ for some $0<c<1.$
\end{remark}

Now we can prove Theorem \ref{t:general_lower_bound}.
\begin{proof}(Theorem \ref{t:general_lower_bound})
From Theorem \ref{t:general_upper_bound} we already know that $\left \|\bar{\xi}_a(1)-x_{a}(t)  \right \|=O \left(\frac{1}{\sqrt{d}} \right)$, so only the lower bound is needed.

\begin{align*}
 \left \|\bar{\xi}_a(1)-x_{a}(1)  \right \| \geq & \E \left( \left | \bar{\xi}_a(1)-x_a(1) \right| \right) \\
 \geq& \E \left( \left |\bar{z}_a(1)-x_a(1) \right |\right)- \E \left(\left | \bar{\xi}_a(1)-\bar{z}_a(1) \right|\right).
\end{align*}

With the same technique as in the proof of Lemma \ref{l:xi_bar_error_decompose} one can show that
\begin{align*}
\E \left(\left | \bar{\xi}_a(1)-\bar{z}_a(1) \right|\right) \leq  \left \| \bar{\xi}_a(1)-\bar{z}_a(1) \right \|= O \left(\max \left \{ \frac{1}{d}, \frac{1}{\sqrt{N}} \right \} \right)    
\end{align*}
 based on Theorem and 4.4 and 4.14 in \cite{keligerConcentrationMeanField2024a}.   

Recall $\mathcal{E}$ from \eqref{eq:events} and the fac that $\pr \left(\mathcal{E}^c \right) $ decays superpolynomial. Using the initial conditions set in Remark \ref{r:Erdos} we obtain
\begin{align*}
   &\E \left( \left |\bar{z}_a(1)-x_a(1) \right |\right) \geq \E \left( \left |\bar{z}_a(1)-x_a(1) \right | \1{\mathcal{E}} \right)\\
   \geq& \frac{e^{-\frac{1}{2}}}{320} \E \left( \frac{\sqrt{(N-1)p(1-p)}}{d} \1{\mathcal{E}}\right)-\underbrace{\frac{1}{2}\E \left( \left \|B-\hat{B} \right \|_{2}^2 \1{\mathcal{E}}\right)}_{=O \left( \frac{1}{d} \right)}-\frac{1}{N} \\
   \geq& \frac{e^{-\frac{1}{2}}}{320} \frac{1}{d} \E \left(\sqrt{(N-1)p(1-p)} \right)-O \left( \frac{1}{d} \right) \\
   =& \frac{e^{-\frac{1}{2}}}{320}  \E \left(\sqrt{\frac{p}{\rho}(1-p)} \right) \frac{1}{\sqrt{d}}-O \left( \frac{1}{d} \right).
\end{align*}

Note that $0 \leq a_{ij} \leq 1$ are independent for $i<j$ and 
\[p=\frac{1}{{\binom{N}{2}}}\sum_{i<j}a_{ij}.\]  
Using Bernstein inequality one can derive that for any $\varepsilon>0$, $\pr \left(|p-\rho|>\varepsilon \right)$ decays superpolynomially. Along with $\rho<1-c$ this implies $\E \left(\sqrt{\frac{p}{\rho}(1-p)} \right) =\Omega \left(1 \right), $ concluding the proof.

\end{proof}

Next, we turn to Theorem \ref{t:homogeneous_lower_bound}.

Take the degree process with initial conditions $\xi_{i,a}(0)=1$. From \eqref{eq:degree_process} $\E_{G}(\xi_{i,a}(t))=z_{i,a}(t)=e^{-\frac{\delta_i}{N \rho}t}$ and $x_{a}(t)=e^{-t}.$

\begin{lemma}
\label{l:moment}
\begin{align}
\label{eq:exp}
\E \left(e^{-\frac{\delta_1}{d}t} \right)=& \exp \left(-d \left(1-e^{-\frac{t}{d}} \right)+O \left(\frac{1}{N} \right) \right) \\
\label{eq:cov}
\E \left(e^{-\frac{\delta_1}{d}t}e^{-\frac{\delta_2}{d}t} \right)=&\exp \left(-2d \left(1-e^{-\frac{t}{d}} \right)+O \left(\frac{1}{N} \right) \right) 
\end{align}
\end{lemma}

\begin{proof} (Lemma \ref{l:moment})

For \eqref{eq:exp} note that $\delta_1 \sim \operatorname{Bin} \left(N-1,\frac{d}{N-1} \right),$ hence, it is enough plug $-\frac{t}{d}$ into the moment generating function.
\begin{align*}
\E \left( e^{-\frac{\delta_1}{d}t} \right)=& \left(1-\frac{d}{N-1} \left(1-e^{-\frac{t}{d}} \right) \right)^{N-1}= \\
\log  \E \left( e^{-\frac{\delta_1}{d}t} \right)=&(N-1) \log \left(1-\frac{d}{N-1} \left(1-e^{-\frac{t}{d}} \right) \right) \\
=& (N-1) \left[-\frac{d}{N-1} \left(1-e^{-\frac{t}{d}} \right)+O \left(\frac{d^2}{(N-1)^2} \underbrace{\left(1-e^{-\frac{t}{d}} \right)^2}_{=O \left( \frac{1}{d^2} \right)} \right)\right] \\
=& -d\left(1-e^{-\frac{t}{d}} \right)+O \left( \frac{1}{N} \right)
\end{align*}

For \eqref{eq:cov} note that $\delta_1$ and $\delta_2$ have distinct terms besides $a_{12}$ making them roughly independent.

\begin{align*}
\E \left( e^{-\frac{\delta_1}{d}t}e^{-\frac{\delta_1}{d}t} \right)=&\E \left[ e^{-\frac{2a_{12}}{d}t} \exp \left(-\frac{t}{d} \sum_{j \neq 1,2} a_{1j} \right) \exp \left(-\frac{t}{d} \sum_{j \neq 1,2} a_{2j} \right)\right] \\
=& \E \left( e^{-\frac{2a_{12}}{d}t} \right)\E^2 \left[  \exp \left(-\frac{t}{d} \sum_{j \neq 1,2} a_{1j} \right) \right]
\end{align*}

The first terms is
\begin{align*}
\E \left( e^{-\frac{2a_{12}}{d}t} \right)=1-\frac{d}{N-1} \left(1-e^{-\frac{2t}{d}} \right)=1+O \left(\frac{1}{N} \right).
\end{align*}

For the second term notice $\sum_{j \neq 1,2}a_{1j} \sim \operatorname{Ber} \left(N-2, \frac{d}{N-1} \right)$ making
\begin{align*}
\E \left[  \exp \left(-\frac{t}{d} \sum_{j \neq 1,2} a_{1j} \right) \right]=& \left(1-\frac{d}{N-1} \left(1-e^{-\frac{t}{d}} \right) \right)^{N-2}\\
=&\exp \left(-d \left(1-e^{-\frac{t}{d}} \right)+O \left(\frac{1}{N} \right) \right) .
\end{align*}

Thus,
\begin{align*}
\E \left( e^{-\frac{\delta_1}{d}t}e^{-\frac{\delta_1}{d}t} \right)=&\left(1+O \left(\frac{1}{N} \right) \right) \left[\exp \left(-d \left(1-e^{-\frac{t}{d}} \right)+O \left(\frac{1}{N} \right) \right)    \right]^2 \\
=& \exp \left(-2d \left(1-e^{-\frac{t}{d}} \right)+O \left(\frac{1}{N} \right) \right) 
\end{align*}
\end{proof}

Now we can finally prove Theorem \ref{t:homogeneous_lower_bound}.

\begin{proof}(Theorem \ref{t:homogeneous_lower_bound} )

\begin{align*}
 \E\left[ \left. \left(\bar{\xi}_a(t)-x_a(t) \right)^2 \right | G \right]=& \mathbb{D}^2 \left( \bar{\xi}_a(t) |G \right)+\left(  \E \left( \bar{\xi}_a(t) | G \right)-x_a(t) \right)^2 \\
 =& \frac{1}{N^2} \sum_{i}e^{-\frac{\delta_i}{d}t}\left(1-e^{-\frac{\delta_i}{d}t} \right)\\
 & + \left( \frac{1}{N}\sum_{i}e^{-\frac{\delta_i}{d}t}-e^{-t} \right)^2 \\
 \|\bar{\xi}_a(t)-x_a(t) \|^2=& \frac{1}{N}\E \left[ e^{-\frac{\delta_1}{d}t}\left(1-e^{-\frac{\
 \delta_1}{d}t} \right) \right] \\
 &+\E \left[ \left( \frac{1}{N}\sum_{i}e^{-\frac{\delta_i}{d}t}-e^{-t} \right)^2 \right]
\end{align*}

The first term is
\begin{align*}
\frac{1}{N}\E \left[ e^{-\frac{\delta_1}{d}t}\left(1-e^{-\frac{\delta_1}{d}t} \right) \right]=&\frac{1}{N} \left[d \underbrace{\left(e^{-\frac{t}{d}}-e^{-\frac{2t}{d}} \right)}_{=\frac{t}{d}+O \left(\frac{1}{d^2} \right)}+O \left( \frac{1}{N} \right) \right]\\
=&\frac{1}{N} \left[t+O \left( \frac{1}{d} \right) \right] = \Theta \left( \frac{1}{N} \right).
\end{align*}

The second:
\begin{align}
\label{eq:error_deg}
\begin{split}
\E \left[ \left( \frac{1}{N}\sum_{i}e^{-\frac{\delta_i}{d}t}-e^{-t} \right)^2 \right]=& \frac{1}{N^2}\sum_{i} \sum_{j}\E \left(e^{-\frac{\delta_i}{d}t} e^{-\frac{\delta_j}{d}t} \right) \\
&-2e^{-t} \E \left(e^{-\frac{\delta_1}{d}t} \right)+e^{-2t}.
\end{split}
\end{align}
In the double sum the diagonal terms contribute at most $O\left( \frac{1}{N} \right)$ amount, hence,
\begin{align*}
\frac{1}{N^2}\sum_{i} \sum_{j}\E \left(e^{-\frac{\delta_i}{d}t} e^{-\frac{\delta_j}{d}t} \right)=&\E \left(e^{-\frac{\delta_1}{d}t} e^{-\frac{\delta_2}{d}t} \right)+O \left(\frac{1}{N} \right)= \\
&\exp \left(-2d \left(1-e^{-\frac{t}{d}} \right)+O \left(\frac{1}{N} \right) \right)+O \left(\frac{1}{N} \right)\\
=&\exp \left(-2d \left(1-e^{-\frac{t}{d}} \right)+O \left(\frac{1}{N} \right) \right).
\end{align*}

This means, \eqref{eq:error_deg} can be written as
\begin{align*}
&\exp \left(-2d \left(1-e^{-\frac{t}{d}} \right)+O \left(\frac{1}{N} \right) \right) \\
& -2e^{-t}\exp \left(-d \left(1-e^{-\frac{t}{d}} \right)+O \left(\frac{1}{N} \right) \right)+e^{-2t}= \\
&e^{-2t} \left[\exp \left(2t-2d \left(1-e^{-\frac{t}{d}} \right)-2 \exp \left(t-d \left(1-e^{-\frac{t}{d}} \right) \right) \right)+1 \right]+O\left( \frac{1}{N} \right)= \\
& e^{-2t}\left[\exp \left(t-d \left(1-e^{-\frac{t}{d}} \right) \right)  -1 \right]^2+O\left( \frac{1}{N} \right)= \\
&e^{-2t} \left(e^{\Theta(\frac{1}{d})}-1 \right)^2+O\left( \frac{1}{N} \right)=\Theta\left(\frac{1}{d^2} \right)+O\left( \frac{1}{N} \right)
\end{align*}

Putting the two terms together results in
\begin{align*}
\|\bar{\xi}_a(t)-x_a(t) \|^2=\Theta \left(\max \left \{\frac{1}{d^2},\frac{1}{N} \right \} \right).
\end{align*}

\end{proof}
\section{Conclusion}\label{sec:conclusion}
In this paper we showed how, for a general class of stochastic processes, the price paid in taking account the fluctuations of the underlying graph model depend on a third layer of randomness, i.e. how the initial conditions are taken. Remarkably, condition \eqref{eq:hom_condition} witness the threshold between two different error regimes, $O\left(\max\left\{1/d,1/\sqrt{n}\right\}\right)$ and $O\left(1/\sqrt{d}\right)$, only depending on how the initial conditions are homogeneously spread around the graph. 

One first interesting question is whether is possible to identify (in the sense of events similar to the one used in \eqref{eq:hom_condition}) interesting families of initial conditions  that can interpolate between the two regimes. Is it possible to find initial condition families that give an error in between the two we obtained for all processes of the type considered here? Furthermore, is it possible to identify the family of initial conditions that give the worst case error? Besides the theoretical aspect, these two questions entail practical consequences in the fashion explained at the beginning of section \ref{subsec:discussion}.

Despite using a class of graph models generic enough, but also able to model communities (which are usually the focus for many of the processes dealing with information spread of epidemics) it would be still of practical interest knowing if our results can be lifted to other type of models. We believe that the restriction imposed on the blocks sizes (they have to be of size $O(N)$) is of technical nature, but not substantial. Similar results should be obtained for rank one random graphs like Chung-Lu and for generic Inhomogeneous Random Graphs as introduced in \cite{bollobasPhaseTransitionInhomogeneous2007}, despite our proofs don't immediately generalize. Even more difficult seems the generalization to graph models were edges become correlated, like Maximal Entropy random graph or Preferential Attachment models.  

Another important effect, noticeable for example in epidemics study, is the fact that, for big outbreaks, the timescale to see a change in the community structure is compatible with the timescale at which the epidemic spreads, especially for endemic diseases. For accounting this influence is necessary to introduce a fourth layer of randomness given by the graph dynamics. Is not hard to believe that for graph dynamics that mix fast enough and are independent form the running stochastic process, the contribution of the fluctuations of the evolving graph structure are already taken in account by our annealing procedure. For example in \cite{hazraFunctionalCentralLimit2024} the largest eigenvalue of a dynamic Erdős-Rényi random graph is shown to follow a functional central limit theorem.  As spectral properties of the underlying random graph model are an ingredient of our results, is fair to conjecture that such gaussian fluctuation can't modify the order of the error. The natural question is therefore if there exists a dynamic graph model (possibly correlated with the running stochastic procecess) whose fluctuations affect the main graphical quantities in a way that the order of the mean-field approximation error needs to be modified.

\printbibliography

\end{document}